\documentclass[12pt]{amsart}
\usepackage{amsmath,amsthm,amsfonts,amssymb,mathrsfs}
\date{\today}

\usepackage{color}
\input xymatrix
\xyoption{all}

\usepackage{hyperref}

\newcommand{\w}{\omega}
\newcommand{\IZ}{\mathbb Z}
\newcommand{\zero}{0}
\newcommand{\one}{\infty}

\newtheorem{theorem}{Theorem}

\newtheorem{proposition}{Proposition}
\newtheorem{corollary}{Corollary}

\theoremstyle{definition}
\newtheorem{example}{Example}
\newtheorem{remark}{Remark}

\newtheorem{definition}{Definition}

\begin{document}

\title[Topological properties of Taimanov semigroups]{Topological properties of Taimanov semigroups}

\author{Oleg~Gutik}
\address{Faculty of Mathematics, National University of Lviv,
Universytetska 1, Lviv, 79000, Ukraine}
\email{o\underline{\hskip5pt}\,gutik@franko.lviv.ua,
ovgutik@yahoo.com}



\begin{abstract}
A semigroup $T$ is called {\em Taimanov} if $T$ contains two distinct elements $\zero,\one$ such that $xy=\one$ for any distinct points $x,y\in T\setminus\{\zero,\one\}$ and $xy=\zero$ in all other cases.  We prove that any Taimanov semigroup $T$ has the following topological properties: (i) each $T_1$-topology with continuous shifts on $T$ is discrete; (ii) $T$ is closed in each $T_1$-topological semigroup containing $T$ as a subsemigroup; (iii) every non-isomorphic homomorphic image $Z$ of $T$ is a zero-semigroup and hence $Z$ is a topological semigroup in any topology on $Z$.
\end{abstract}

\maketitle

We shall follow the terminology of~\cite{Carruth-Hildebrant-Koch-1983-1986, Clifford-Preston-1961-1967, Engelking-1989, Ruppert-1984}. 

The problem of non-discrete (Hausdorff) topologization of infinite groups was posed by Markov \cite{Markov-1945}. 
This problem was resolved by Ol'shanskiy \cite{Olshansky-1980} who constructed an infinite countable group $G$ admitting no non-discrete Hausdorff group topologies. On the other hand, Zelenyuk \cite{Zel} proved that each group $G$ admits a non-discrete shift-continuous Hausdorff topology $\tau$ with continuous inversion $G\to G$, $x\mapsto x^{-1}$. In \cite[2.10]{BPS} it was observed that Ol'shanskiy construction can be modified to produce for every non-zero $m\in\IZ\setminus\{-2^{n},2^n:n\in\w\}$ a countable infinite group $G_m$ admitting no non-discrete shift-continuous topology with continuous $m$-th power map $G_m\to G_m$, $x\mapsto x^m$.

Studying the topologizability problem in the class of inverse semigroups, Eberhart and Selden  \cite{Eberhart-Selden-1969} proved that every Hausdorff semigroup topology on the bicyclic semigroup $\mathscr{C}(p,q)$ is discrete. This result was generalized by Bertman and West \cite{Bertman-West-1976} who proved that every Hausdorff shift-continuous topology on $\mathscr{C}(p,q)$ is discrete. In \cite{Bardyla-2016,Bardyla-Gutik-2016,Chuchman-Gutik-2010,Chuchman-Gutik-2011,Fihel-Gutik-2011, {Gutik-2015}, Gutik-Maksymyk-2016??, Gutik-Pozdnyakova-2014, Gutik-Repovs-2011, Gutik-Repovs-2012, Mesyan-Mitchell-Morayne-Peresse-2016} these topologizability results were extended to some generalizations of the bicyclic semigroup.

Studying the topologizability problem in the class of commutative semigroups \cite{Taimanov-1975}, Taimanov in \cite{Taimanov-1973} constructed a commutative semigroup $\mathfrak A_\kappa$ of arbitrarily large cardinality $\kappa$, which admits no non-discrete Hausdorff semigroup topology, but any non-isomorphic homomorphic image $Z$ of $T$ is a zero-semigroup and hence is a topological semigroup in any topology on $Z$. We recall that a semigroup $Z$ is a {\em zero-semigroup} if the set $SS=\{xy:x,y\in X\}$ is a singleton $\{z\}$. In this case the element $z$ is the {\em zero-element} of the semigroup $S$, i.e., a (unique) element $z\in S$ such that $xz=z=zx$ for al $x\in S$. In this paper we improve the mentioned Taimanov's result proving that the Taimanov semigroup $\mathfrak A_\kappa$ admits no non-discrete shift-continuous $T_1$-topologies and is closed in any $T_1$-topological semigroup containing $\mathfrak A_\kappa$ as a subsemigroup. First we give an abstract definition of a Taimanov semigroup.

\begin{definition} {\rm A semigroup $T$ is called {\em Taimanov} if it contains two distinct elements $\zero_T,\one_T$  such that for any $x,y\in T$
$$x\cdot y=\begin{cases}\one_T&\mbox{if $x\ne y$ and $x,y\in T\setminus\{\zero_T,\one_T\}$};\\
\zero_T&\mbox{if $x=y$ or $\{x,y\}\cap\{\zero_T,\one_T\}\ne\emptyset$}.
\end{cases}
$$The elements $\zero_T,\one_T$ are uniquely determined by the algebraic structure of $T$: $\zero_T$ is a (unique) zero-element of $T$, and $\one_T$ is the unique element of the set $TT\setminus\{\zero_T\}$.
}\end{definition}

It follows that each Taimanov semigroup $T$ is commutative.
Concrete examples of Taimanov semigroups can be constructed as follows.

\begin{example} {\rm For any non-zero cardinal $\kappa$ the set $\kappa\cup\{\kappa\}$ endowed with the commutative semigroup operation defined by
$$xy=\begin{cases}\kappa&\mbox{if $x\ne y$ and $x,y\in T\setminus\{\zero,\kappa\}$};\\
0&\mbox{if $x=y$ or $\{x,y\}\cap\{\zero,\kappa\}\ne\emptyset$}.
\end{cases}
$$ is a Taimanov semigroup of cardinality $1+\kappa$. Here we identify the cardinal $\kappa$ with the set $[0,\kappa)$ of ordinals, smaller than $\kappa$.}
\end{example}

\begin{proposition} Two Taimanov semigroups are isomorphic if and only if they have the same cardinality.
\end{proposition}

\begin{proof} Given two Taimanov semigroups $T,S$ of the same cardinality, observe that any bijective map
$f:T\to S$ with $f(\zero_T)=\zero_S$ and  $f(\one_T)=\one_S$ is an algebraic isomorphism of $T$ onto $S$.
\end{proof}

In this paper we show that any Taimanov semigroup $T$ has the following topological properties:
\begin{itemize}\itemsep=1pt\parskip1pt
  \item[(1)] every shift-continuous $T_1$-topology on $T$ is discrete;
  \item[(2)] $T$ is closed in each $T_1$-topological semigroup containing $T$ as a subsemigroup;
  \item[(3)] every non-isomorphic homomorphic image $Z$ of $T$ is a zero-semigroup and hence any topology on $Z$ turns it into a topological semigroup.
\end{itemize}

The first statement generalizes the original result of Taimanov \cite{Taimanov-1973} and is proved in the following proposition.

\begin{proposition}\label{proposition-3} Every shift-continuous $T_1$-topology $\tau$ on any Taimanov semigroup $T$ is discrete.
\end{proposition}

\begin{proof} The statement is trivial if the semigroup $T$ is finite. So, assume that $T$ is infinite. The topology $\tau$ satisfies the separation axiom $T_1$ and hence contains an open set $U\subset X$ such that $\zero_T\in U$ and $\one_T\notin U$.

First we prove that the points $\zero_T$ and $\one_T$ are isolated in $T$. Chose any point $x\in T\setminus\{\zero_T,\one_T\}$ and observe that $x\cdot \zero_T=x\cdot \one_T=\zero_T\in U$. By the shift-continuity of the topology $\tau$, there exist neighborhoods $U_0\in\tau$ of $\zero_T$ and $U_{\!\one}\in\tau$ of $\one_T$ such that $(x\cdot U_0)\cup (x\cdot U_{\!\one})\subset U$. We claim that $U_0\setminus\{x,\one_T\}=\{\zero_T\}$ and $U_{\!\one}\setminus\{x,\zero_T\}=\{\one_T\}$. In the opposite case we could find a point $y\in (U_0\cup U_{\!\one})\setminus\{x,\zero_T,\one_T\}$ and conclude that $\one_T=xy\in x\cdot(U_0\cup U_{\!\one})\subset U\subset T\setminus\{\one_T\}$, which is a desired contradiction showing that the points $\zero_T$ and $\one_T$ are isolated in $T$.

To show that each point $x\in T\setminus\{\zero_T,\one_T\}$ is isolated in the topology $\tau$, observe that $xx=\zero_T\in T\setminus \{\one_T\}\in\tau$ and use the shift-continuity of the topology $\tau$ to find a neighborhood $U_x\in\tau$ of $x$ such that $xU_x\subset T\setminus\{\one_T\}$. Assuming that $U_x\ne\{x\}$ we can choose any point $y\in U_x\setminus\{x\}$ and conclude that $\one_T=xy\in T\setminus\{\one_T\}$, which is a contradiction showing that $U_x=\{x\}$ and hence the point $x$ is isolated in the topology $\tau$.
\end{proof}

The following example shows that any infinite Taimanov semigroup admits a non-discrete semigroup $T_0$-topology.

\begin{example}\label{examle-4} {\rm For any infinite Taimanov semigroup $T$ the family of subsets
\vskip5pt

$
\centerline{$\tau:=\big\{U\subset T:$ if $\zero_T\in U$, then $\one_T\in U$ and  $|T\setminus U|<\w\big\}$}
$
\vskip5pt

\noindent is a $T_0$-topology turning $T$ into a topological semigroup.}
\end{example}

 A semitopological semigroup $S$ will be called {\em square-topological} if the map $S\to S$, $x\mapsto x^2$, is continuous. It is clear that each topological semigroup is square-topological.

\begin{theorem}\label{theorem-5} A Taimanov semigroup $T$ is closed in any square-topological semigroup $S$ containing $T$ as a subsemigroup and satisfying the separation axiom $T_1$.
\end{theorem}

\begin{proof} Assuming that $T$ is not closed in $S$, choose any point $s\in\bar T\setminus T$. We claim that $sx=\one_T$ for any $x\in T\setminus\{\zero_T,\one_T\}$. Assuming that $sx\ne \one_T$ and using  the shift-continuity of the $T_1$-topology of $S$, we can find a neighborhood $U_s\subset S$ of $s$ such that $U_s\cdot x\subset S\setminus\{\one_T\}$. Since $s$ is an accumulation point of the set $T$ in $S$, there exists a point $y\in U_s\setminus\{x,\zero_T,\one_T\}$. For this point $y$ we get $\one_T=yx\in U_sx\subset S\setminus\{\one_T\}$, which is a contradiction showing that $sx=\one_T$ for any $x\in T\setminus\{\zero_T,\one_T\}$. Next, we show that $ss=\one_T$. Assuming that $ss\ne \one_T$, we can use the shift-continuity of the $T_1$-topology of $S$ to find a neighborhood $V_s\subset S$ of $s$ such that $sV_s\subset S\setminus\{\one_T\}$. Since $s$ is an accumulation point of the set $T$ in $S$, there exists a point $x\in V_s\cap T\setminus\{\zero_T,\one_T\}$. For this point $x$, we get $\one_T=sx\in sV_s\subset S\setminus\{\one_T\}$, which is a contradiction showing that $ss=\one_T$. By the separation axiom $T_1$, the set $S\setminus\{\zero_T\}$ is an open neighborhood of $\one_T$ in $S$. The continuity of the map $S\to S$, $x\mapsto x^2$, yields a neighborhood $W_s\subset S$ such that $x^2\in S\setminus\{\zero_T\}$ for any $x\in W_s$. Since $s$ is an accumulation point of the set $T$ in $S$, there exists a point $x\in W_s\cap T\setminus\{\zero_T,\one_T\}$. For this point $x$ we get $\zero_T=xx\in S\setminus\{\zero_T\}$, which is a desired contradiction showing that the set $T$ is closed in $S$.
\end{proof}

The following example shows that any infinite Taimanov semigroup admits a (non-closed) embedding into a compact Hausdorff semitopological semigroup and also shows that the continuity of the map $S\to S$, $x\mapsto x^2$, in Theorem~\ref{theorem-5} is essential and cannot be replaced by the continuity of the map $S\to S$, $x\mapsto x^m$, for some $m\ge 3$.

\begin{example}\label{examle-6} {\rm Let $T$ be a Taimanov semigroup and $X$ be any $T_1$-topological space containing $T$ as a non-closed dense discrete subspace. Extend the semigroup operation of $T$ to a binary operation of $X$ defined by the formula:
$$xy=\begin{cases}
\zero_T&\mbox{if $x=y\in T$ or $\{x,y\}\cap\{\zero_T,\one_T\}\ne\emptyset$};\\
\one_T&\mbox{otherwise}.
\end{cases}
$$
Since $(xy)z=\zero_T=x(yz)$ for any $x,y,z\in X$ the extended operation is associative and turns $X$ into a commutative semigroup containing $T$ as a subsemigroup.
Observe that for $a\in\{\zero_T,\one_T\}$ the shift $l_a=r_a:X\to X$, $x\mapsto ax=xa=0_T$, is constant and hence continuous. For any $a\in T\setminus\{\zero_T,\one_T\}$ the shift $l_a=r_a:X\to X$, $x\mapsto xa=ax$, is almost constant in the sense that $l_a^{-1}(\one_T)=X\setminus\{a,\zero_T,\one_T\}$ and hence is continuous (as the set $\{a,\zero_T,\one_T\}$ is closed and open in $X$).
For any $a\in X\setminus T$ the shift $l_a=r_a:X\to X$, $x\mapsto xa=ax$, is almost constant in the sense that $l_a^{-1}(\one_T)=X\setminus\{\zero_T,\one_T\}$ and hence is continuous.
This shows that $X$ is a semitopological commutative semigroup containing $T$ as a non-closed dense subsemigroup. Observe also that for every $m\ge 3$ the map $X^m\to X$, $(x_1,\dots,x_m)\mapsto x_1\cdots x_m=\zero_T$, is constant and hence continuous. Then the map $X\to X$, $x\mapsto x^m$, is continuous as well.}
\end{example}

\begin{example}\label{examle-7} {\rm For any topological zero-semigroup $Z$ with zero $\zero_Z$ and any Taimanov semigroup $T$ endowed with the discrete topology, any map $h:T\to Z$ with $h(\zero_T)=h(\one_T)=\zero_Z$ is a continuous semigroup homomorphism. Hence there exist many topological (zero-)semigroups containing continuous homomorphic images of Taimanov semigroups as non-closed subsemigroups.}
\end{example}

\begin{proposition}\label{proposition-8} Any non-isomorphic homomorphic image $S$ of a Taimanov semigroup $T$ is a zero-semigroup.
\end{proposition}

\begin{proof} Fix a non-injective surjective homomorphism $h:T\to S$.
If $f(\zero_T)=f(\one_T)$, then $SS=f(T)\cdot f(T)=f(TT)=f(\{\zero_T,\one_T\})=\{f(\zero_T)\}$, which means that $S$ is a zero-semigroup. So, assume that $f(\zero_T)\ne f(\one_T)$. Since $f$ is not injective, there exist two distinct points $a,b\in T$ with $f(a)=f(b)$. Since $f(\zero_T)\ne f(\one_T)$, one of the points $a,b$, say $a$, belongs to $T\setminus\{\zero_T,\one_T\}$. If $b\notin\{\zero_T,\one_T\}$, then $ab=\one_T$ and $aa=\zero_T$ and hence $f(\one_T)=f(ab)=f(a)f(b)=f(a)f(a)=f(aa)=f(\zero_T)$, which contradicts our assumption. This contradiction shows that $b\in\{\zero_T,\one_T\}$ and hence $bc=\zero_T$ for any $c\in T$.

If $|T|\ge4$, then we can find a point $c\in T\setminus\{a,\zero_T,\one_T\}$ and conclude that $f(\one_T)=f(ac)=f(a)f(c)=f(b)f(c)=f(bc)=f(\zero_T)$, which contradicts our assumption.
So, $|T|\le 3$ and hence $T=\{a,\zero_T,\one_T\}$ and
\vskip4pt

\centerline{
$S=f(T)=\{f(a),f(\zero_T),f(\one_T)\}=\{f(b),\{f(\zero_T),f(\one_T)\}=\{f(\zero_T),f(\one_T)\}.$}
\vskip4pt

\noindent Then  $SS=f(\{xy:x,y\in\{\zero_T,\one_T\}\})=\{f(\zero_T)\}$, which means that $S$ is a zero-semigroup.
\end{proof}

Since the semigroup operation $Z\times Z\to\{\zero_Z\}\subset Z$ of any zero-semigroup $Z$ is constant and hence is continuous with respect to any topology on $X$, Proposition~\ref{proposition-8} implies the following corollary.

\begin{corollary}\label{corollary-9}
Every non-isomorphic homomorphic image $S$ of a Taimanov semigroup is a topological semigroup with respect to any topology on $S$.
\end{corollary}

We call that a semigroup $S$ is \emph{algebraically complete} in a class $\mathscr{S}$ of semitopological semigroups if $S$ is a closed subsemigroup in each semitopological semigroup $T\in\mathscr S$ containing $S$ as a subsemigroup. Theorem~\ref{theorem-5} implies the following

\begin{corollary}\label{corollary-10} Each Taimanov semigroup $T$ is algebraically complete in the class of square-topological semigroups satisfying the separation axiom $T_1$. In particular, $T$ is algebraically complete in the class of $T_1$-topological semigroups.
\end{corollary}

\begin{remark}
Corollary~\ref{corollary-9} implies that for any Taimanov semigroup $T$ and any non-iso\-morphic surjective homomorphism $h:T\to S$ with the infinite image $S=h(T)$ the semigroup $S$ is a dense proper subsemigroup of some (compact) Hausdorff topological zero-semigroup.
\end{remark}

\medskip
\noindent{\bf Acknowledgements.} We acknowledge Taras Banakh and the referee for useful important comments and suggestions.


\end{document}